\documentclass[10pt,reqno]{amsart}
\usepackage[utf8]{inputenc}
\usepackage[english]{babel}
\usepackage{amsmath, amssymb, amsthm}
\usepackage{mathrsfs}
\usepackage{enumitem}
\usepackage{graphicx}
\usepackage{tikz-cd}
\usepackage{xcolor}
\numberwithin{equation}{section} 

\usepackage{diagbox} 
\usepackage{tensor}  

\usepackage{hyperref} 

\usepackage[alphabetic]{amsrefs} 

\theoremstyle{plain}
\newtheorem{thm}{Theorem}[section]

\newtheorem{lemma}[thm]{Lemma}

\theoremstyle{definition}

\newtheorem{rmk}[thm]{Remark}

\allowdisplaybreaks

\newcommand{\N}{\mathbb{N}}

\newcommand{\eps}{\varepsilon}

\DeclareMathOperator{\Div}{div}
\DeclareMathOperator{\ric}{Ric}

\DeclareMathOperator{\tr}{tr}

\def\XXint#1#2#3{{\setbox0=\hbox{$#1{#2#3}{\int}$ }
\vcenter{\hbox{$#2#3$ }}\kern-.6\wd0}}

\usepackage{scalerel,stackengine}
\stackMath
\newcommand\hhat[1]{%
\savestack{\tmpbox}{\stretchto{%
  \scaleto{%
    \scalerel*[\widthof{\ensuremath{#1}}]{\kern.1pt\mathchar"0362\kern.1pt}%
    {\rule{0ex}{\textheight}}
  }{\textheight}%
}{2.4ex}}%
\stackon[-6.9pt]{#1}{\tmpbox}%
}

\title[Inverse Scattering for AH manifolds and the Conformal Laplacian]{On the interplay between inverse scattering for asymptotically hyperbolic manifolds and the Calderón problem for the Conformal Laplacian}
\author{Sebastián Muñoz-Thon}
\address{Universit\'e Paris-Saclay, Laboratoire de math\'ematiques d’Orsay, 91405, Orsay, France.}
\email{sebastian.munoz-thon@universite-paris-saclay.fr}

\begin{document}

\begin{abstract}
In this short note, we use the relation obtained by Guillarmou--Guillopé and Chang--González between the generalized eigenvalue problem for asymptotically hyperbolic (AH) manifolds and the Conformal Laplacian, to obtain a new inverse scattering result: on an AH manifold of dimension $n+1$ with constant scalar curvature $-n(n+1)$, we show that the scattering matrix at energy $\frac{n+1}{2}$ determines the jet of the metric on the boundary, up to a diffeomorphism and conformal factor.
\end{abstract}

\maketitle

\section{Introduction}

Let $\overline{X}$ be a compact manifold with boundary of dimension $n+1$. A smooth metric $g$ in the interior $X$ of $\overline{X}$ is \emph{conformally compact} (CC) if $\overline{g}=\rho^{2} g$ extends smoothly to a Riemannian metric in $\overline{X}$, where $\rho \in C^{\infty}(\overline{X},[0,\infty))$ is a boundary defining function for $\partial \overline{X}$ (i.e., $d\rho|_{\partial \overline{X}} \neq 0$ and $\{\rho=0\}=\partial \overline{X}$). The boundary $\partial \overline{X}$ equipped with the conformal class of $\overline{g}|_{T \partial \overline{M}}$ is called the \emph{conformal boundary} or \emph{conformal infinity} of $(\overline{X}, g)$. If $|d\rho|_{\overline{g}}=1$ at the boundary, $g$ is called \emph{asymptotically hyperbolic} (AH). A more restrictive condition is asking for the Einstein condition $\ric_{g}=-ng$. In this case, $g$ is called \emph{conformally compact Einstein} (CCE or \emph{Poincar\'e--Einstein}).

In these particular ``edge'' geometries (in the sense of \cite{Mazzeo91_ell}), we study the inverse scattering problem for the generalized eigenvalue problem
\begin{equation} \label{eq:eigen_prob}
    -\Delta_{g}u-s(n-s)u=0 \quad \text{in }X,
\end{equation}
where $s=\frac{n+1}{2}$. If $(\overline{X},g)$ is an AH manifold with constant scalar curvature $R_{g}\equiv -n(n+1)$, given $f \in C^{\infty}(\partial \overline{X})$ the problem \eqref{eq:eigen_prob} has a unique solution $u$ with $u|_{\partial \overline{X}}=f$ and the expansion
\[ u=\rho^{n-s}(f+O(\rho^{3}))+\rho^{s}\left( -S\left( \frac{n+1}{2} \right)+O(\rho)\right), \]
see \cites{GZ03, Guillarmou05, GG07} and Lemma \ref{lemma:asymp_ex} below. $S$ (sometimes written as $S_{g}$ to make clear the dependence on $g$) is the \emph{scattering matrix at energy $\frac{n+1}{2}$}, and depends on the choice of the boundary defining function (or equivalently, of the choice of the metric on the conformal infinity), although it can be defined without this dependence, see \cite{GZ03}*{Section 3}). Hereafter, we denote by $P_{s}$ the operator appearing in \eqref{eq:eigen_prob}.

As is mentioned in \cites{MM87, GZ03}, the scattering matrix can be considered as a version of the Dirichlet-to-Neumann (DN) for these geometries. It then makes sense to consider the analog of the Calder\'on problem to this context. In inverse problems, this means to determine the metric (or even the differential or topological structure) of a compact Riemannian manifold with boundary, from the DN map associated to the Laplace--Beltrami operator. On compact Riemannian manifolds, this is solved on surfaces and in the analytic category \cites{LU01}, and in \emph{conformally anisotropic geometries} \cite{DSFKSU09}. The Calder\'on problem in our context, known as the inverse scattering problem, asks if it is possible to recover information of the manifold $(\overline{X},g)$ from the knowledge of the scattering matrix $S_{g}(s)$, for certain fixed $s$. In this direction, our main result is the following:

\begin{thm} \label{thm:jet}
Let $(\overline{X}_{i},g_{i})$ be asymptotically hyperbolic manifolds of dimension $n+1$, with constant scalar curvature $R_{g_{i}} \equiv -n(n+1)$. Suppose that $\partial \overline{X}_{i}$ contain a common open set $\Gamma$ such that the identity map $id \colon \Gamma \subset \partial \overline{X}_{i} \to \Gamma \subset \overline{X}_{2}$ is a diffeomorphism. Assume that $\ker_{L^{2}}P_{\frac{n+1}{2}}^{i}=\emptyset$ for $i=1,2$. If for some boundary defining functions $\rho_{i}$, the scattering matrices at energy $\frac{n+1}{2}$ agree for the two metrics, i.e., $S_{g_{1}}(\frac{n+1}{2})=S_{g_{2}}(\frac{n+1}{2})$, then there exists $c\in C^{\infty}(\overline{X}_{1})$ such that: 
\begin{enumerate}
    \item when $n=1$, there exists a diffeomorphism $\psi \colon \overline{X}_{1} \to \overline{X}_{2}$ with $\psi|_{\Gamma}=id_{\Gamma}$ and such that $c\psi^* g_{2}=g_{1}$;
    \item when $n\geq 2$, there exists a diffeomorphism $\psi \colon U_{1} \to U_{2}$ between neighborhoods $U_{i}$ of $\Gamma$ in $\overline{X}_{i}$, with $\psi|_{\Gamma}=id_{\Gamma}$ and such that $c\psi^* g_{2}$ and $g_{1}$ coincide up to infinite order.
\end{enumerate}
\end{thm}

The result is inspired by the classical result for compact Riemannian manifolds with boundary obtained in \cite{LU89}. In contrast to the proof given in \cite{LU89}, which is based on a factorization of the Laplace--Beltrami operator, our proof follows an indirect approach (although we also have a ``formal'' factorization of the Laplacian in AH geometries, see Lemma \ref{lemma:fac}): we use the relation between the Dirichlet problem for \eqref{eq:eigen_prob} at $s=\frac{n+1}{2}$ and the Dirichlet problem for the Conformal Laplacian (also known as the Yamabe operator), described first in \cite{GG07} for AH manifolds with $R_{g} \equiv -n(n+1)$, and then in \cite{CG11} for more general equations in CCE manifolds. Note that this means we have conformal invariance for the scattering operator this energy, so the conformal factors appearing in the result are natural. They show that, for manifolds as in Theorem \ref{thm:jet}, the DN map of $\rho^{2}g$ is the scattering of $g$ at energy $\frac{n+1}{2}$. According to the results of \cite{LLS22}, the DN map for the Conformal Laplacian determines the jet of a metric at the boundary, which is conformal to $\rho^{2}g$. We remark, however, that the determination of the jet on the boundary in \cite{LLS22} relies on a factorization of the magnetic Schr\"odinger operator carried out in \cite{DSFKSU09}.

Although the theorem has a short proof, the result is new since there are no inverse scattering results at this energy level, and furthermore, we impose no restrictions on the dimension. To the authors' knowledge, the first inverse scattering result in the AH manifold context was given in \cite{JSB00}, where it was shown that the scattering matrix on AH manifolds associated with \eqref{eq:eigen_prob} with a potential determines the metric up to diffeomorphism at infinite order, assuming the same potential, and also the potential up to infinite order if one fixes the metric. They results do not apply in our case due to their restrictions on the energy levels. Using radiation fields, it was proven in \cite{SB05} that the scattering matrix at \emph{all} energies determines the metric up to a diffeomorphism on AH manifolds. In the same spirit, inverse scattering results for asymptotically complex hyperbolic manifolds were obtained in \cite{GSB08}. On even dimensional CCE manifolds (of dimension $n+1$), it is known \cite{GSB09} that the scattering matrix at energy $n$ determines the metric up to a diffemorphism. On CC manifolds, it was shown in \cite{Marazzi09} that the scattering matrix at two energy levels, associated with \eqref{eq:eigen_prob} with potential, determines the potential and the metric (and even the curvature at the boundary) up to the usual gauge. In \cite{GT11}, an inverse scattering result for the Schr\"odinger equation on AH manifolds was obtained, showing that the scattering matrix determines the metric, assuming that the potentials vanish up to order 2 in the boundary. Below, we mention an updated version of this result, which also recovers the metric, see Theorem \ref{thm:2d}. The results from \cite{GSB09} were extended in \cite{Marazzi11}, where it is shown that on CCE manifolds, the scattering matrix at energy $n$ determines the metric at certain order involving $\log$ terms, up to diffeomorphism, and a modified scattering operator determines the metric up to diffeomorphism. In \cite{Isozaki04}, inverse scattering problems in the hyperbolic space were studied and related, similarly to this work, to inverse problems for the Schr\"odinger operator in Euclidean domains. In \cite{IKL10}, inverse scattering was consider for manifolds with cylindrical ends, while in \cite{IKL17}, surfaces with conical singularities, cusps and regular ends were considered, see also the monographs \cites{IK14, IKL14, IKL15}. The results in \cite{SB05} were extended to partial data in \cite{HSB15}, and for scattering with disjoint source and observation sets in \cite{HSB18}. We also highlight the work \cite{Marx-Kuo24} on AH manifolds, where, assuming some evenness conditions on the metric, it is shown that one can determine it from the knowledge of minimal areas.

\subsection*{Acknowledgments}

The author would like to thank Colin Guillarmou, Rafe Mazzeo, Mikko Salo, Gunther Uhlmann, and Andr\'as Vasy for helpful discussions. Additionally, the author expresses gratitude to Stanford University and the University of Washington for their hospitality during his visit, while this work was in progress. The author was partly supported by the Ross-Lynn Scholar Grant from Purdue University's College of Science.

\section{Proof of the Main Result} \label{section:proof}

Before dealing with the inverse problem, we show that the direct problem has a solution with asymptotics as mentioned in the introduction. The most fundamental work in this direction is \cite{GZ03}. The following is mentioned in \cite{GG07} as a consequence of \cite{Guillarmou05}*{Lemma 4.1}, but we provide a short proof by completeness.

\begin{lemma} \label{lemma:asymp_ex}
Let $(\overline{X},g)$ be an asymptotically hyperbolic manifold with $R_{g} \equiv -n(n+1)$. Assume that $\ker_{L^{2}}P_{\frac{n+1}{2}}=\emptyset$. Given $f \in C^{\infty}(\partial \overline{X})$, there exists a unique solution $u$ to
\[ P_{\frac{n+1}{2}}u=0, \quad u|_{\partial \overline{X}}=f, \]
having the expansion
\[ u=\rho^{\frac{n-1}{2}}(f+O(\rho^{3}))+\rho^{\frac{n+1}{2}}\left( -S\left( \frac{n+1}{2} \right)+O(\rho)\right). \]
The term of order $\frac{n+1}{2}$, i.e., $S((n+1)/2)$ is defined as the \emph{scattering matrix of energy $(n+1)/2$}.
\end{lemma}

\begin{proof}
We follow \cites{GZ03, Guillarmou05,GSB09}. The strategy is as follows: First, we construct a solution modulo infinite order. Then, the hypothesis on the kernel and the metric allow us to apply \cite{Guillarmou05}*{Lemma 4.1} to ensure that the resolvent has no poles at $(n+1)/2$. Therefore, we can modify the just constructed solution to obtain a full solution having the required asymptotics.

The first step is to construct $F_{\infty} \in C^{\infty}(\overline{X})$ such that 
\begin{equation} \label{eq:F_infty}
    P_{\frac{n+1}{2}} \rho^{\frac{n-1}{2}} F_{\infty}=O(\rho^{\infty}), \quad F_{\infty}|_{\rho=0}=f.
\end{equation}
where $\rho$ is a Graham--Lee geodesic boundary defining function (see \cite{GL91}*{Lemma 5.2} and the subsequent paragraph, see also \cite{Graham00}*{Lemma 2.1}). As mentioned in \cite{Guillarmou05}*{p. 24}, it is enough to work in a neighborhood of the boundary. Using the flow $\phi_t(y)$ of the gradient $\nabla^{\rho^2 g} \rho$ of the Graham--Lee boundary defining function, we obtain a diffeomorphism $\phi \colon [0, \eps)_{t} \times \partial \overline{X} \to \phi([0, \eps) \times \partial \overline{X}) \subset \overline{X}$ defined by $\phi(t, y):=\phi_t(y)$, and the metric pulls back to
\[ \phi^{*} g=\frac{d t^2+h(t)}{t^2}, \]
for some smooth one-parameter family of metrics $h(t)$ on the boundary $\partial \overline{X}$, where $\phi^* \rho=t$. 
It can be shown that
\[ P_{\frac{n+1}{2}} t^{\frac{n-1}{2}}=t^{\frac{n-1}{2}} \mathscr{D}, \]
where
\[ \mathscr{D}:=-t^{2}\partial_{x}^{2}-\frac{t}{2}\tr_{h}\partial_{t}h t\partial_{t}-\frac{n-1}{4}\tr_{h}\partial_{t}h+t^{2}\Delta_{h}. \]
Observe that for $j \in \N_{0}=\{0,1,\ldots\}$ we have
\begin{equation} \label{eq:ind_id}
    \mathscr{D} (f t^j )=j(1-j) f t^j+t^j G\left(\frac{n+1}{2}-j\right) f,
\end{equation}
where
\[ (G(z)f)(x,y):=t^{2}\Delta_{h}f-\frac{n-z}{2}t\tr_{h}\partial_{t}h f. \]
Suppose now that $\tilde{F} \in C^{\infty}(\overline{X})$ satisfies $\mathscr{D}\tilde{F}=O(t^{j})$ for $j \geq 1$. Since $G(z)f=O(t)$ (a priori, we analyze this further below), using \eqref{eq:ind_id} with $f=(t^{-j}\mathscr{D}\tilde{F})|_{t=0}/(j(1-j))$, we obtain 
\begin{equation} \label{eq:construction_F_1_x^j}
    \mathscr{D}\tilde{F}-\mathscr{D}\left( t^{j} \frac{(t^{-j}\mathscr{D}\tilde{F})|_{t=0}}{j(1-j)} \right)=O(t^{j+1}),
\end{equation}
for $j \neq 1$. Indeed, on the left hand side of \eqref{eq:construction_F_1_x^j}, we are subtracting the term of order $O(t^{j})$ from $\mathscr{D}\tilde{F}$. This procedure suggests the following recursive definition of $f_{j}$ and $F_{j}$:
\begin{equation} \label{eq:ind_constr}
    F_0:=f_0, \quad f_j :=\frac{-(t^{-j} \mathscr{D} (F_{j-1}))|_{t=0}}{j(1-j)}, \quad F_j :=F_{j-1}+f_j t^j, \quad j \geq 4 .
\end{equation}
To construct $f_{1}$, $f_{2}$, and $f_{3}$, we proceed as follows: To construct $f_{1}$, and hence $F_{1}$, we look up for terms or order $1$ in $\mathscr{D}F_{0}$. First, recall from \cite{GG07}*{Section 4} that $R_{g} \equiv -n(n+1)$ implies $\tr_{h_{0}}h_{1}=0$. Indeed, since $g$ is conformal to $t^{2}\overline{g}=dt^{2}+h(t)$, we have
\[ R_{g}=-n(n+1)+nt\partial_{t}\log (\det h(t))+t^{2}R_{\overline{g}}. \]
Since $R_{g} \equiv -n(n+1)$, after simplification, we obtain, by first diving by $t$ and then letting $t \to 0$, that $\tr_{h_{0}}h_{1}=0$, where $h_{0}=h(0)$ and $h_{1}=\partial_{t}h(0)$. This proves the claim. Now, we are in good shape to use \cite{Guillarmou05}*{Equations (4.8), (4.9)}, that since $\tr_{h_{0}}h_{1}=0$, we have 
\[ G(z)=t^{2}(\Delta_{h_{0}}+D^{0})+t^{3}D^{2}+O(t^{4}), \]
where $D^{j}$ are differential operators of order $j$ on $\partial \overline{X}$. So $\mathscr{D}f$ only contains terms of order $3$ or more. Hence, $f_{1}:=0$ and $F_{1}:=F_{0}$. Similarly, for $f_{2}$ we look up for terms or order $2$ in $\mathscr{D}F_{1}$. By the previous argument, there are non of them, and we set $f_{2}:=0$ and $F_{2}:=F_{0}$. Finally, for $f_{3}$ we see that $t^{3}(\Delta_{h_{0}}+D^{2})f$ is the term of order 3 in $\mathscr{D}F_{2}=\mathscr{D}f$, and hence we define $f_{3}:=(\Delta_{h_{0}}+D^{2})f$ and $F_{3}=t^{3}f_{3}+f_{0}$. We can now continue using \eqref{eq:ind_constr} inductively. By Borel's lemma, we can find $F_{\infty} \in C^{\infty}(\overline{X})$ such that $F_{\infty}-F_{j}=O(t^{j+1})$, and satisfying \eqref{eq:F_infty}.

To complete the construction of the full solution to the eigenvalue problem, we now use our assumption on $g$ and the $L^{2}$ kernel to invoke \cite{Guillarmou05}*{Lemma 4.1}, which shows that the resolvent of $P_{s}$ has no poles at $s=(n+1)/2$. Hence, it is an analytic family by \cite{GZ03}*{Proposition 3.1} (and the results in \cite{MM87}), and the resolvent of $P_{\frac{n+1}{2}}$ satisfies
\[ R\left( \frac{n+1}{2} \right) \colon \dot{C}^{\infty}(\overline{X}) \to \rho^{\frac{n+1}{2}}C^{\infty}(\overline{X}), \]
where $\dot{C}^{\infty}(\overline{X}):=\{u \in C^{\infty}(\overline{X}):u=O(\rho^{\infty})\}$. Therefore, we can set
\[ u=\rho^{\frac{n-1}{2}}F_{\infty}-R\left( \frac{n+1}{2} \right)\rho^{\frac{n-1}{2}}F_{\infty} \]
which solves the eigenvalue problem. 

To see the uniqueness, we notice that if $u_{1}$ and $u_{2}$ are two solutions with the same term of order $\rho^{\frac{n-1}{2}}$, then $u_{1}-u_{2} \in \rho^{\frac{n+1}{2}}C^{\infty}(\overline{X})$. This implies that $u_{1}-u_{2} \in L^{2}(\overline{X})$, contradicting our hypothesis. This finishes the proof.    
\end{proof}

\begin{rmk}
    Alternatively, one can use a more direct method in the following sense. One can obtain a solution to the problem $P_{s}u'_{s}=O(\rho^{\infty})$ with $u'_{s}|_{\partial \overline{X}}=f$ for $s$ close to $(n+1)/2$. Then, a full solution can be obtained by using the resolvent at energy $s$. Since there are no poles, $R(s)$ extends to $R((n+1)/2)$, and one then obtain a solution $u$ to $P_{\frac{n+1}{2}}u=0$ and $u_{|\partial \overline{X}}=f$.
\end{rmk}

\begin{rmk}
    We note that we only used the curvature condition to ensure that $\tr_{h_{0}}h_{1}=0$. The vanishing of the trace means that $\partial \overline{X}$ minimal in $(\overline{X},\rho^{2}g)$. In particular, the result is also true if we change the curvature assumption by evenness of the metric $g$ modulo $O(\rho^{3})$. We recall that $(M,g)$ is \emph{even modulo} $O(\rho^{2k+1})$ if there is a boundary defining function $\rho$ with $|d\rho|_{\rho^{2}g}=1$, and such that in a collar neighborhood, the pullback of $g$ by the gradient flow of $\nabla^{\rho^{2}g}\rho$ takes the form
    \[ \frac{1}{t^{2}}\left(dt^2+\sum_{i=0}^k h_{2 i} t^{2 i}+h_{2 k+1} t^{2 k+1}+O(t^{2 k+2})\right), \]
    where $h_{j}$'s are symmetric tensors on $\partial \overline{X}$. It was shown in \cite{Guillarmou05}*{Lemma 2.1} that this is a well-defined notion, meaning that it is independent of the boundary defining function. Examples of even metrics can be found in the remark after Theorem 1.4 in \cite{Guillarmou05} and the references therein. For instance, hyperbolic metrics perturbed on a compact set are even modulo $O(\rho^{\infty})$, and CCE manifolds of dimension $n+1$ are even modulo $O(\rho^{n})$.
\end{rmk}

The key relation between inverse scattering on AH manifolds and the Calder\'on problem for the Conformal Laplacian is given by the following result of Guillarmou and Guilollopé:

\begin{lemma}[\cite{GG07}*{Section 4}] \label{lemma_equiv}
Let $(\overline{X},g)$ be an asymptotically hyperbolic manifold, with $R_{g} \equiv -n(n+1)$. Then $u \in C^{\infty}(\overline{X})$ is a solution to 
\[ P_{\frac{n+1}{2}}u=0, \quad u|_{\partial \overline{X}}=f, \]
if and only if $U:=\rho^{-\frac{n-1}{2}}u$ is a solution to 
\[ \left(-\Delta_{\overline{g}}+\frac{n-1}{4n}R_{\overline{g}}\right)U=0, \quad U|_{\partial \overline{X}}=f, \]
where $R_{\overline{g}}$ is the scalar curvature associated to $\overline{g}$. Furthermore, 
\[ S \left( \frac{n+1}{2} \right)=\Lambda_{\overline{g}}, \]
as operators from $C^{\infty}(\partial \overline{X})$ to itself.
\end{lemma}

We point out that on CCE manifolds, this was extended to any energy level in \cite{CG11}*{Lemma 4.1, Theorem 4.3}. 

To reach our goal, we will need the following result concerning the boundary jet recovery in the inverse problem for the Conformal Laplacian.

\begin{lemma}[\cite{LLS22}*{Lemmas 2.1 \& 2.2}] \label{lemma:jet_conf}
Let $(\overline{X}_{1},\overline{g}_{1})$, $(\overline{X}_{2},\overline{g}_{2})$ be compact Riemannian manifolds with boundary, with a common boundary portion $\Gamma \subset \partial \overline{X}_{1}, \partial \overline{X}_{2}$. Assume that Dirichlet-to-Neumann maps for the Conformal Laplacian agree, i.e., $\Lambda_{\overline{g}_{1}}=\Lambda_{\overline{g}_{2}}$. Then, there exist $c_{i} \in C^{\infty}(M_{i})$ positive functions with $c_{i}|_{\Gamma}=1$ such that if $\tilde{g}_{i}=c_{i}\overline{g}_{i}$, then the jets of $\tilde{g}_{1}$ and $\tilde{g}_{2}$ on the respective $\tilde{g}_{i}$-normal coordinates agree.
\end{lemma}

\begin{proof}[Proof of Theorem \ref{thm:jet}]
    Let $\rho_{i}$ be geodesic boundary defining functions for $g_{i}$ such that the scattering matrices at energy $\frac{n+1}{2}$ coincide, Lemma \ref{lemma_equiv} implies that the same holds for the DN maps associated to the Conformal Laplacian of the metrics $\rho_{1}^{2}g_{1}$ and $\rho_{2}^{2}g_{2}$. If $n \geq 2$, Lemma \ref{lemma:jet_conf} says that there exist $c_{i} \in C^{\infty}(\overline{X}_{i})$ with $c_{i}|_{\Gamma}=1$, $i=1,2$ such that given coordinates $y$ in $\Gamma$, if $\psi_1$ and $\psi_2$ are $c_{1}\rho_{1}^{2}g_1$- and $c_{2}\rho_{2}^{2}g_2$-boundary normal coordinates constructed by using the same $y$-coordinates, then $(\psi_{1}^{-1})^{*}(c_{1}\rho_{1}^{2}g_{1})=(\psi_{2}^{-1})^{*}(c_{2}\rho_{2}^{2}g_{2})+O(t^{\infty})$. Hence, $g_{1}=(c_{1}\rho_{1}^{2})^{-1}(\psi_{1}\circ \psi_{2}^{-1})^{*}(c_{2}\rho_{2}^{2}g_{2})+O(\rho_{1}^{\infty})$. To deal with surfaces, one applies the result from \cite{LU01} instead of \cite{LLS22} to conclude the proof.
\end{proof}

\begin{rmk}
Note that for surfaces, the hypothesis in the kernel is superfluous by the results in \cites{MM87, GZ03}.
\end{rmk}

\begin{rmk}
By the results in \cite{GG07}, one actually obtains equivalence between the inverse scattering problem at energy $1$ and the anisotropic Calder\'on problem for surfaces. However, in higher dimensions, one also needs information of the mean curvature of the boundary with respect to the compactified metric.
\end{rmk}

\begin{rmk}
When $n+1 \geq 2$, if all the components are analytic, i.e., the metrics, the conformal factors and the diffeomorphisms, then one can extend the diffemorphism as in \cite{LU89} to obtain a global one. Furthermore, if one assumes that the compactified manifolds $(\overline{X}_{i},c_{i}g_{i})$ are \emph{locally conformally real analytic} in the sense of \cite{LLS22}, then a global result can be achieved as in the case of surfaces (Theorem \ref{thm:2d}). We point out that is condition seems to be stronger that being analytic on the interior, so it is not immediate that CCE are locally conformally real analytic up to the boundary. We also mention that since we recover metrics conformal to the initial ones, even if the initial ones are CCE, they conformal versions are in general not Einstein. Hence, we cannot apply the methods of \cite{GSB08} in our case to obtain a global result.
\end{rmk}

We would like to point out that one needs the conformal factors in order to obtain a vanishing condition in a certain expression involving the metric tensor in order to recover the jet of the metric at the boundary. The recovery of the jet is obtained by the following procedure, see \cite{DSFKSU09}*{Equation 8.2}, \cite{LLS22}*{p.~11136} (We point out that there is a little typo there, the condition should be \[ 
\partial_{\tilde{x}_n}^j (\tilde{g}_{\alpha \beta} \partial_{\tilde{x}_n} \tilde{g}^{\alpha \beta}) (x', 0)=0, \quad j \geq 1.) \] First, one factorizes the Conformal Laplacian, and then one shows that the Dirichlet-to-Neumann map is equal to a pseudodifferential operator of order 1 appearing in such factorization, modulo a smoothing operator. Hence, to recover the jet of the metric, one shows that the DN map determines this $\Psi$DO, and finally one shows that from the full symbol one can recover the jet of the metric. In the context of inverse problems, this technique was first used in \cite{SU88} for the conductivity equation, \cite{LU89} for the Laplace--Beltrami operator, and then \cite{NSU95} for the magnetic Schr\"odinger operator with Euclidean metric. In \cite{DSFKSU09} the procedure was generalized for the Schr\"odinger operator in the presence of a magnetic field on Riemannian manifolds. We observe that if in particular we have no magnetic field this procedure gives a (local) factorization of $P_{s}$. Indeed, in the coordinates given by the geodesic boundary defining function $\rho$ (i.e., on $\partial \overline{X} \times [0,\eps)$, we use coordinates $(y,t)$ so that $\overline{g}=\rho^{2}g=dt^{2}+h_{t}(y)$, with dual variables $(\xi,\eta)$), we have
\begin{equation} \label{eq:fac_cl}
    L_{\overline{g}}=(D_{t}+iE-iB(y,t,D_{y}))(D_{t}+iB(y,t,D_{y}))
\end{equation}
modulo $\Psi^{-\infty}$, where
\[ E(x)=\frac{1}{2}\overline{g}_{\alpha \beta}\partial_{t}\overline{g}^{\alpha \beta}, \]
and $B \in \Psi^{1}$ has symbol $b\sim \sum_{j \leq 1}b_{j}$, where each term in the expansion is given by 
\begin{align*}
    b_{0} &=-\sqrt{q_{2}}, \\
    b_{1} &=\frac{1}{2 b_1}(-\partial_t b_1+E b_1-\nabla_{\eta} b_1 \cdot D_{y} b_1+q_1), \\
    b_{-1} &=\frac{1}{2 b_1}\left(-\partial_t b_0+E b_0-\sum_{\substack{0 \leq j, k \leq 1 \\ j+k=|K|}} \frac{\partial_{\eta}^K b_j D_{y}^K b_k}{K!}+\frac{n-1}{4n}R_{\overline{g}}\right), \\
    b_{m-1} &=\frac{1}{2 b_1}\left(-\partial_t b_m+E b_m-\sum_{\substack{m \leq j, k \leq 1 \\ j+k-|K|=m}} \frac{\partial_{\eta}^K b_j D_{y}^K b_k}{K!}\right) \quad(m \leq-1).
\end{align*}
Here, $q_{2}$ and $q_{1}$ are the symbols of 
\[ \overline{g}^{\alpha \beta}D_{\alpha}D_{\beta}, \qquad -i\left(\frac{1}{2} \overline{g}^{\alpha \beta} \partial_\alpha(\log |g|)+\partial_\alpha \overline{g}^{\alpha \beta}\right) D_\beta, \]
respectively. Here the left quantization is used. Hence, using the invariance of the Conformal Laplacian together with \eqref{eq:fac_cl}, we obtain
\begin{align*}
    P_{\frac{n+1}{2}}&=t^{\frac{n+3}{2}}L_{\overline{g}} t^{-\frac{n-1}{2}} \\
    &= t^{\frac{n+3}{2}}(D_{t}+iE-iB(y,t,D_{y}))(D_{t}+iB(y,t,D_{y}))t^{-\frac{n-1}{2}} \\
    &=t^{\frac{n+3}{2}}(D_{t}+iE-iB(y,t,D_{y}))t^{-\frac{n+1}{2}}\left(tD_{t}-\frac{n-1}{2}+itB \right) \\
    &=\left( tD_{t}-\frac{n+1}{2}+itE-itB \right) \left(tD_{t}-\frac{n-1}{2}+itB \right).
\end{align*}

We summarize this discussion as follows:

\begin{lemma} \label{lemma:fac}
    Formally, we have the factorization
    \[P_{\frac{n+1}{2}}=\left( tD_{t}-\frac{n+1}{2}+itE-itB \right)\left( tD_{t}-\frac{n-1}{2}+itB \right) \]
    modulo $t \Psi^{-\infty}$.
\end{lemma}

We refer to this as a formal factorization, since it is merely an algebraic manipulation, but not a proper factorization in the standard algebra of $\Psi$DO. Since $\Delta_{g}$ is a 0-differential operator (\cite{Mazzeo91_ell}*{Proposition 2.6}), one should expect to have a factorization in the 0-calculus of Mazzeo and Melrose. Lemma \ref{lemma:fac} suggests in general a factorization of the following form
\begin{equation} \label{eq:factorization}
    P_{s}=\left( tD_{t}-s+iE-iA \right)\left( tD_{t}-(n-s)+iA \right).
\end{equation}

\begin{rmk}
We point out that it is shown in \cite{CG11} that there exists ``special'' boundary defining functions $\rho_{1}$, $\rho_{2}$ so that, in a collar neighborhood of the boundary, $P_{s}$ becomes:
\begin{enumerate}
    \item the Laplace--Beltrami operator associated to $\rho_{1}^{2}g$, when $s=\frac{n+1}{2}$,
    \item the Bakry--\'Emery Laplacian (also known as Ornstein-Uhlenbeck operator, Witten Laplacian or Weighted Laplacian)  $-\Div ((\rho_{2})^{a} \nabla \bullet)$. Inverse problems for this operator have been studied recently in \cite{BK25}. The equation can be thought as applying the operator 
    \[ L=e^{V}\Div (e^{-V} \nabla \bullet), \]
    where $V=-\log\rho_{2}^{a}$.
\end{enumerate}
\end{rmk}

Finally, we state the following stronger version of Theorem \ref{thm:jet} in dimension 2.

\begin{thm} \label{thm:2d}
Let $(\overline{X},g_{i})$ be two asymptotically hyperbolic surfaces with common boundary $\partial \overline{X}$, and let $V_{i}\in \rho_{i}^2 C^{\infty}(\overline{X})$ be two potentials. Assume that
\[ \{\partial_{\rho_{1}} u |_{\partial \overline{X}} : u \in \ker_{L^2}(\Delta_{g_{1}}+V_1 )\}=\{\partial_{\rho_{2}} u |_{\partial \overline{X}} : u \in \ker_{L^2}(\Delta_{g_{2}}+V_2 )\}, \]
and let $S_{i}(1)$ (associated to $\rho_{i}$) be the scattering map for the operator $\Delta_{g_{i}}+V_i$ for $i=1, 2$. If $S_{1}(1)=S_{2}(1)$, then there exists a diffeomorphism $\psi \colon \overline{X} \to \overline{X}$ fixing the boundary pointwise, and a smooth function $c$ such that
\[ g_{1}=c\psi^{*}g_{2}, \quad V_{1}=c^{-1}\psi^{*}V_{2}. \]
\end{thm}

The proof of \cite{GT11}*{Corollary 6.1} applies mutatis mutandis: it is not hard to see that the DN maps for the operators $\Delta_{\overline{g}_{i}}+\overline{V}_i$ are the same, then \cite{CLT24}*{Theorem 1.1} allows one to recover both the metrics and the potentials with the gauge mentioned in the result. The difference with \cite{GT11}*{Corollary 6.1} lies in the fact that previously only the potential could be recovered, as an application of the main theorem of \cite{GT11}.


\begin{bibdiv} 
\begin{biblist}

\bib{BK25}{article}{
   author={Borthwick, Jack},
   author={Kamran, Niky},
   title={Inverse problems for the Bakry-\'Emery Laplacian on manifolds with boundary - uniqueness and non-uniqueness},
   date={2025},
   eprint={2503.21944},
   status={preprint},
}

\bib{CLT24}{article}{
   author={C\^{a}rstea, C\u{a}t\u{a}lin I.},
   author={Liimatainen, Tony},
   author={Tzou, Leo},
   title={The Calderón problem on Riemannian surfaces and of minimal surfaces},
   date={2024},
   eprint={2406.16944},
   status={preprint},
}

\bib{CG11}{article}{
   author={Chang, Sun-Yung Alice},
   author={Gonz\'alez, Mar\'ia del Mar},
   title={Fractional Laplacian in conformal geometry},
   journal={Adv. Math.},
   volume={226},
   date={2011},
   number={2},
   pages={1410--1432},
   issn={0001-8708},
   review={\MR{2737789}},
   doi={10.1016/j.aim.2010.07.016},
}

\bib{DSFKSU09}{article}{
   author={Dos Santos Ferreira, David},
   author={Kenig, Carlos E.},
   author={Salo, Mikko},
   author={Uhlmann, Gunther},
   title={Limiting Carleman weights and anisotropic inverse problems},
   journal={Invent. Math.},
   volume={178},
   date={2009},
   number={1},
   pages={119--171},
   issn={0020-9910},
   review={\MR{2534094}},
   doi={10.1007/s00222-009-0196-4},
}

\bib{Graham00}{article}{
   author={Graham, C. Robin},
   title={Volume and area renormalizations for conformally compact Einstein
   metrics},
   booktitle={The Proceedings of the 19th Winter School ``Geometry and
   Physics'' (Srn\'i, 1999)},
   journal={Rend. Circ. Mat. Palermo (2) Suppl.},
   number={63},
   date={2000},
   pages={31--42},
   issn={1592-9531},
   review={\MR{1758076}},
}

\bib{GL91}{article}{
   author={Graham, C. Robin},
   author={Lee, John M.},
   title={Einstein metrics with prescribed conformal infinity on the ball},
   journal={Adv. Math.},
   volume={87},
   date={1991},
   number={2},
   pages={186--225},
   issn={0001-8708},
   review={\MR{1112625}},
   doi={10.1016/0001-8708(91)90071-E},
}

\bib{GZ03}{article}{
   author={Graham, C. Robin},
   author={Zworski, Maciej},
   title={Scattering matrix in conformal geometry},
   journal={Invent. Math.},
   volume={152},
   date={2003},
   number={1},
   pages={89--118},
   issn={0020-9910},
   review={\MR{1965361}},
   doi={10.1007/s00222-002-0268-1},
}

\bib{Guillarmou05}{article}{
   author={Guillarmou, Colin},
   title={Meromorphic properties of the resolvent on asymptotically
   hyperbolic manifolds},
   journal={Duke Math. J.},
   volume={129},
   date={2005},
   number={1},
   pages={1--37},
   issn={0012-7094},
   review={\MR{2153454}},
   doi={10.1215/S0012-7094-04-12911-2},
}

\bib{GG07}{article}{
   author={Guillarmou, Colin},
   author={Guillop\'e, Laurent},
   title={The determinant of the Dirichlet-to-Neumann map for surfaces with boundary},
   journal={Int. Math. Res. Not. IMRN},
   date={2007},
   number={22},
   pages={Art. ID rnm099, 26},
   issn={1073-7928},
   review={\MR{2376211}},
   doi={10.1093/imrn/rnm099},
}

\bib{GSB08}{article}{
   author={Guillarmou, Colin},
   author={S\'a{} Barreto, Ant\^onio},
   title={Scattering and inverse scattering on ACH manifolds},
   journal={J. Reine Angew. Math.},
   volume={622},
   date={2008},
   pages={1--55},
   issn={0075-4102},
   review={\MR{2433611}},
   doi={10.1515/CRELLE.2008.064},
}

\bib{GSB09}{article}{
   author={Guillarmou, Colin},
   author={S\'a{} Barreto, Ant\^onio},
   title={Inverse problems for Einstein manifolds},
   journal={Inverse Probl. Imaging},
   volume={3},
   date={2009},
   number={1},
   pages={1--15},
   issn={1930-8337},
   review={\MR{2558301}},
   doi={10.3934/ipi.2009.3.1},
}

\bib{GT11}{article}{
   author={Guillarmou, Colin},
   author={Tzou, Leo},
   title={Calder\'on inverse problem with partial data on Riemann surfaces},
   journal={Duke Math. J.},
   volume={158},
   date={2011},
   number={1},
   pages={83--120},
   issn={0012-7094},
   review={\MR{2794369}},
   doi={10.1215/00127094-1276310},
}

\bib{HSB15}{article}{
   author={Hora, Raphael},
   author={S\'a{} Barreto, Ant\^onio},
   title={Inverse scattering with partial data on asymptotically hyperbolic manifolds},
   journal={Anal. PDE},
   volume={8},
   date={2015},
   number={3},
   pages={513--559},
   issn={2157-5045},
   review={\MR{3353825}},
   doi={10.2140/apde.2015.8.513},
}

\bib{HSB18}{article}{
   author={Hora, Raphael},
   author={S\'a{} Barreto, Ant\^onio},
   title={Inverse scattering with disjoint source and observation sets on
   asymptotically hyperbolic manifolds},
   journal={Comm. Partial Differential Equations},
   volume={43},
   date={2018},
   number={9},
   pages={1363--1376},
   issn={0360-5302},
   review={\MR{3915490}},
   doi={10.1080/03605302.2018.1517793},
}

\bib{Isozaki04}{article}{
   author={Isozaki, Hiroshi},
   title={Inverse spectral problems on hyperbolic manifolds and their
   applications to inverse boundary value problems in Euclidean space},
   journal={Amer. J. Math.},
   volume={126},
   date={2004},
   number={6},
   pages={1261--1313},
   issn={0002-9327},
   review={\MR{2102396}},
}

\bib{IK14}{book}{
   author={Isozaki, Hiroshi},
   author={Kurylev, Yaroslav},
   title={Introduction to spectral theory and inverse problem on
   asymptotically hyperbolic manifolds},
   series={MSJ Memoirs},
   volume={32},
   publisher={Mathematical Society of Japan, Tokyo},
   date={2014},
   pages={viii+251},
   isbn={978-4-86497-021-1},
   review={\MR{3222614}},
   doi={10.1142/e040},
}

\bib{IKL10}{article}{
   author={Isozaki, Hiroshi},
   author={Kurylev, Yaroslav},
   author={Lassas, Matti},
   title={Forward and inverse scattering on manifolds with asymptotically cylindrical ends},
   journal={J. Funct. Anal.},
   volume={258},
   date={2010},
   number={6},
   pages={2060--2118},
   issn={0022-1236},
   review={\MR{2578464}},
   doi={10.1016/j.jfa.2009.11.009},
}

\bib{IKL14}{article}{
   author={Isozaki, Hiroshi},
   author={Kurylev, Yaroslav},
   author={Lassas, Matti},
   title={Recent progress of inverse scattering theory on non-compact
   manifolds},
   conference={
      title={Inverse problems and applications},
   },
   book={
      series={Contemp. Math.},
      volume={615},
      publisher={Amer. Math. Soc., Providence, RI},
   },
   isbn={978-1-4704-1079-7},
   date={2014},
   pages={143--163},
   review={\MR{3221603}},
   doi={10.1090/conm/615/12290},
}

\bib{IKL15}{article}{
   author={Isozaki, Hiroshi},
   author={Kurylev, Yaroslav},
   author={Lassas, Matti},
   title={Inverse scattering on multi-dimensional asymptotically hyperbolic orbifolds},
   conference={
      title={Spectral theory and partial differential equations},
   },
   book={
      series={Contemp. Math.},
      volume={640},
      publisher={Amer. Math. Soc., Providence, RI},
   },
   isbn={978-1-4704-0989-0},
   date={2015},
   pages={71--85},
   review={\MR{3381017}},
   doi={10.1090/conm/640/12840},
}

\bib{IKL17}{article}{
   author={Isozaki, Hiroshi},
   author={Kurylev, Yaroslav},
   author={Lassas, Matti},
   title={Conic singularities, generalized scattering matrix, and inverse
   scattering on asymptotically hyperbolic surfaces},
   journal={J. Reine Angew. Math.},
   volume={724},
   date={2017},
   pages={53--103},
   issn={0075-4102},
   review={\MR{3619104}},
   doi={10.1515/crelle-2014-0076},
}

\bib{JSB00}{article}{
   author={Joshi, Mark S.},
   author={S\'a{} Barreto, Ant\^onio},
   title={Inverse scattering on asymptotically hyperbolic manifolds},
   journal={Acta Math.},
   volume={184},
   date={2000},
   number={1},
   pages={41--86},
   issn={0001-5962},
   review={\MR{1756569}},
   doi={10.1007/BF02392781},
}

\bib{LLS22}{article}{
   author={Lassas, Matti},
   author={Liimatainen, Tony},
   author={Salo, Mikko},
   title={The Calder\'on problem for the conformal Laplacian},
   journal={Comm. Anal. Geom.},
   volume={30},
   date={2022},
   number={5},
   pages={1121--1184},
   issn={1019-8385},
   review={\MR{4564032}},
   doi={10.4310/cag.2022.v30.n5.a6},
}

\bib{LU01}{article}{
   author={Lassas, Matti},
   author={Uhlmann, Gunther},
   title={On determining a Riemannian manifold from the Dirichlet-to-Neumann map},
   language={English, with English and French summaries},
   journal={Ann. Sci. \'Ecole Norm. Sup. (4)},
   volume={34},
   date={2001},
   number={5},
   pages={771--787},
   issn={0012-9593},
   review={\MR{1862026}},
   doi={10.1016/S0012-9593(01)01076-X},
}

\bib{LU89}{article}{
   author={Lee, John M.},
   author={Uhlmann, Gunther},
   title={Determining anisotropic real-analytic conductivities by boundary
   measurements},
   journal={Comm. Pure Appl. Math.},
   volume={42},
   date={1989},
   number={8},
   pages={1097--1112},
   issn={0010-3640},
   review={\MR{1029119}},
   doi={10.1002/cpa.3160420804},
}

\bib{Marazzi09}{article}{
   author={Marazzi, Leonardo},
   title={Inverse scattering on conformally compact manifolds},
   journal={Inverse Probl. Imaging},
   volume={3},
   date={2009},
   number={3},
   pages={537--550},
   issn={1930-8337},
   review={\MR{2557918}},
   doi={10.3934/ipi.2009.3.537},
}

\bib{Marazzi11}{article}{
   author={Marazzi, Leonardo},
   title={Asymptotically hyperbolic manifolds with polyhomogeneous metric},
   journal={Differential Integral Equations},
   volume={24},
   date={2011},
   number={9-10},
   pages={973--1000},
   issn={0893-4983},
   review={\MR{2850349}},
}

\bib{Marx-Kuo24}{article}{
   author={Marx-Kuo, Jared},
   title={An Inverse Problem for Renormalized Area: Determining the Bulk Metric with Minimal Surfaces},
   date={2024},
   eprint={2401.07394},
   status={preprint},
}

\bib{Mazzeo91_ell}{article}{
   author={Mazzeo, Rafe},
   title={Elliptic theory of differential edge operators. I},
   journal={Comm. Partial Differential Equations},
   volume={16},
   date={1991},
   number={10},
   pages={1615--1664},
   issn={0360-5302},
   review={\MR{1133743}},
   doi={10.1080/03605309108820815},
}

\bib{Mazzeo91_uni}{article}{
   author={Mazzeo, Rafe},
   title={Unique continuation at infinity and embedded eigenvalues for
   asymptotically hyperbolic manifolds},
   journal={Amer. J. Math.},
   volume={113},
   date={1991},
   number={1},
   pages={25--45},
   issn={0002-9327},
   review={\MR{1087800}},
   doi={10.2307/2374820},
}

\bib{MM87}{article}{
   author={Mazzeo, Rafe R.},
   author={Melrose, Richard B.},
   title={Meromorphic extension of the resolvent on complete spaces with
   asymptotically constant negative curvature},
   journal={J. Funct. Anal.},
   volume={75},
   date={1987},
   number={2},
   pages={260--310},
   issn={0022-1236},
   review={\MR{0916753}},
   doi={10.1016/0022-1236(87)90097-8},
}

\bib{NSU95}{article}{
   author={Nakamura, Gen},
   author={Sun, Zi Qi},
   author={Uhlmann, Gunther},
   title={Global identifiability for an inverse problem for the
   Schr\"odinger equation in a magnetic field},
   journal={Math. Ann.},
   volume={303},
   date={1995},
   number={3},
   pages={377--388},
   issn={0025-5831},
   review={\MR{1354996}},
   doi={10.1007/BF01460996},
}

\bib{SB05}{article}{
   author={S\'a{} Barreto, Ant\^onio},
   title={Radiation fields, scattering, and inverse scattering on
   asymptotically hyperbolic manifolds},
   journal={Duke Math. J.},
   volume={129},
   date={2005},
   number={3},
   pages={407--480},
   issn={0012-7094},
   review={\MR{2169870}},
   doi={10.1215/S0012-7094-05-12931-3},
}

\bib{SU88}{article}{
   author={Sylvester, John},
   author={Uhlmann, Gunther},
   title={Inverse boundary value problems at the boundary---continuous
   dependence},
   journal={Comm. Pure Appl. Math.},
   volume={41},
   date={1988},
   number={2},
   pages={197--219},
   issn={0010-3640},
   review={\MR{0924684}},
   doi={10.1002/cpa.3160410205},
}

\end{biblist}
\end{bibdiv}
\end{document}